\documentclass[a4paper,10pt]{article}

\usepackage[english]{babel}
\usepackage{amsmath,amsfonts,amssymb,amsthm}
\usepackage{mathrsfs}
\usepackage{graphicx}
\usepackage{bbm}
\usepackage{indentfirst}

\newtheorem{theorem}{Theorem}[section]

\newtheorem{lemma}{Lemma}[section]

\newtheorem{corollary}{Corollary}[section]

\theoremstyle{definition}
\newtheorem{remark}{Remark}[section]

\numberwithin{equation}{section}

\newcommand\blfootnote[1]{\begingroup\renewcommand\thefootnote{}\footnote{#1}\addtocounter{footnote}{-1}\endgroup}

\DeclareTextCompositeCommand{\u}{PD1}{\i}{i}

\begin{document}

\title{
{\bf\Large
Existence of positive solutions of a superlinear boundary value problem with indefinite weight }\footnote{Work performed under the auspices of the
Grup\-po Na\-zio\-na\-le per l'Anali\-si Ma\-te\-ma\-ti\-ca, la Pro\-ba\-bi\-li\-t\`{a} e le lo\-ro
Appli\-ca\-zio\-ni (GNAMPA) of the Isti\-tu\-to Na\-zio\-na\-le di Al\-ta Ma\-te\-ma\-ti\-ca (INdAM).}}

\author{
{\bf\large Guglielmo Feltrin}
\vspace{1mm}\\
{\it\small SISSA - International School for Advanced Studies}\\
{\it\small via Bonomea 265}, {\it\small 34136 Trieste, Italy}\\
{\it\small e-mail: guglielmo.feltrin@sissa.it}\vspace{1mm}}

\date{}

\maketitle

\vspace{-2mm}

\begin{abstract}
\noindent
We deal with the existence of positive solutions for a two-point boundary value problem associated with the nonlinear second order equation
\linebreak
$u''+a(x)g(u)=0$. The weight $a(x)$ is allowed to change its sign. 
We assume that the function $g\colon\mathopen{[}0,+\infty\mathclose{[}\to\mathbb{R}$ is continuous, 
$g(0)=0$ and satisfies suitable growth conditions, so as the case $g(s)=s^{p}$, with $p>1$, is covered. 
In particular we suppose that $g(s)/s$ is large near infinity, but we do not require that
$g(s)$ is non-negative in a neighborhood of zero.
Using a topological approach based on the Leray-Schauder degree we obtain a result of
existence of at least a positive solution that improves previous existence theorems.
\end{abstract}

\section{Introduction}\label{section1}

In this paper we are interested in the study of positive solutions for the nonlinear two-point boundary value problem
\begin{equation}\label{two-pointBVPag}
\begin{cases}
\, u''+a(x)g(u)=0 \\
\, u(0)=u(L)=0,
\end{cases}
\end{equation}
where $a\colon\mathopen[0,L\mathclose]\to{\mathbb{R}}$ is a Lebesgue integrable function 
and $g\colon{\mathbb{R}}^{+}\to{\mathbb{R}}$ is a continuous function, 
where ${\mathbb{R}}^{+}:=\mathopen{[}0,+\infty\mathclose{[}$ denotes the set of non-negative real numbers.
We recall that a \textit{positive solution} of \eqref{two-pointBVPag} is an absolutely continuous function 
$u\colon\mathopen[0,L\mathclose]\to\mathbb{R}^{+}$ such that its derivative $u'(x)$ is absolutely continuous, 
$u(x)$ satisfies \eqref{two-pointBVPag} for a.e.~$x\in\mathopen[0,L\mathclose]$ 
and $u(x)>0$ for every $x\in\mathopen]0,L\mathclose[$.

This issue has been considered by many authors. As classical examples, we mention 
\cite{AmannLopezGomez1998,BerestyckiCapuzzo-DolcettaNirenberg1994,BonheureGomesHabets2005,ErbeWang1994,LanWebb1998,Nussbaum1975} (see also the references therein),
where different techniques are used to face this type of problem. Our work benefits from a new approach 
based on the Leray-Schauder topological degree,
so, to obtain a positive solution, our goal is to prove that the degree of a suitable operator is non-zero on 
an open domain of $\mathcal{C}(\mathopen{[}0,L\mathclose{]})$ not containing the trivial solution.

Our assumptions allow the weight function $a(x)$ to change its sign a finite number of times and, concerning the nonlinearity,
we suppose that $g(s)$ can change its sign, even an infinite number of times, and that, roughly speaking, it has a superlinear growth
at zero and at infinity. More in detail, with respect to the growth of $g(s)/s$ at zero, we assume a very general condition 
which depends on the sign of $g(s)$ in a right neighborhood of zero.

Our main result states that, under the conditions just presented, problem \eqref{two-pointBVPag} has at least a positive solution.
This theorem clearly covers the case $g(s)=s^{p}$, with $p>1$. Moreover, the results concerning the BVP \eqref{two-pointBVPag} 
where is assumed that $a(x)g(s)\geq0$ for a.e.~$x\in\mathopen[0,L\mathclose]$ and for all $s\geq0$  (see~\cite{ErbeWang1994,LanWebb1998,Nussbaum1975}) 
or that $g(s)>0$ for all $s>0$, 
when $a(x)$ is allowed to change sign (see~\cite{BonheureGomesHabets2005,FeltrinZanolin2015,GaudenziHabetsZanolin2003}), do not contain our result and, in some cases, are easy consequences of it.

Figure~\ref{figure1} and Figure~\ref{figure2} show examples of nonlinearities $g(s)$ satisfying our assumptions and which are not covered
by previous results.

\begin{figure}[h!]
\centering
\includegraphics[width=0.38\textwidth]{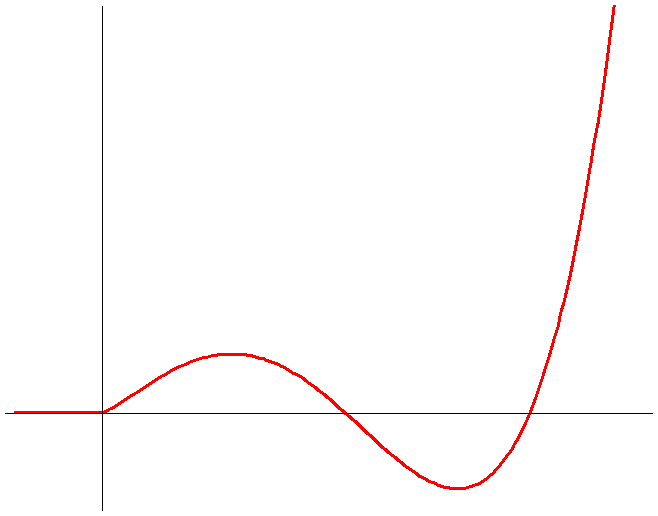}\quad\quad\quad
\includegraphics[width=0.38\textwidth]{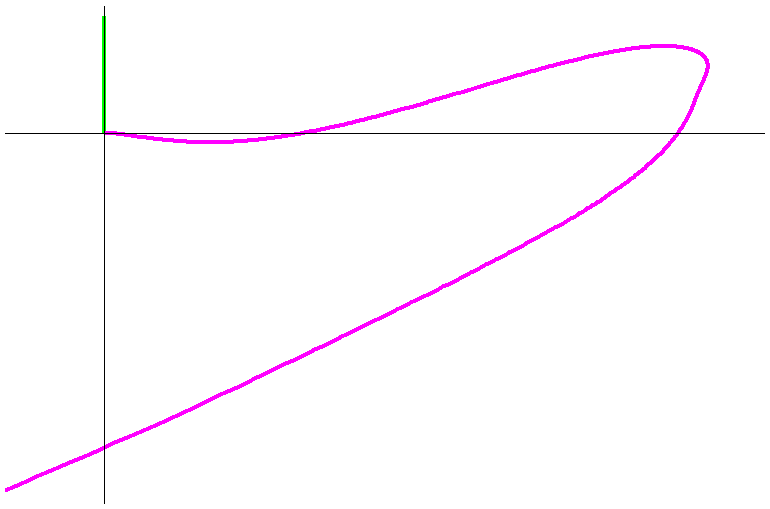}
\caption{\small{A numerical simulation obtained by setting
$I = \mathopen[0,1\mathclose]$, $a(x) = \sin(3\pi x)$
and $g(s) = \min\{20s^{6/5}-6s^3+s^4, 400\, s\arctan(s)\}$. 
On the left we have shown the graph of $g(s)$. We underline that $g(s)$ changes sign and $g(s)/s\not\to +\infty$ as $s\to+\infty$.
On the right we have represented the image of
the segment $\{0\}\times \mathopen[0,12\mathclose]$ through the Poincar\'{e} map in the phase-plane $(u,u')$.
It intersects the negative part of the $u'$-axis in a point, hence there is a positive initial slope at $x=0$
from which departs a solution which is positive on $\mathopen]0,1\mathclose[$ and vanishes at $x=1$.}}
\label{figure1}
\end{figure}

\begin{figure}[h!]
\centering
\includegraphics[width=0.38\textwidth]{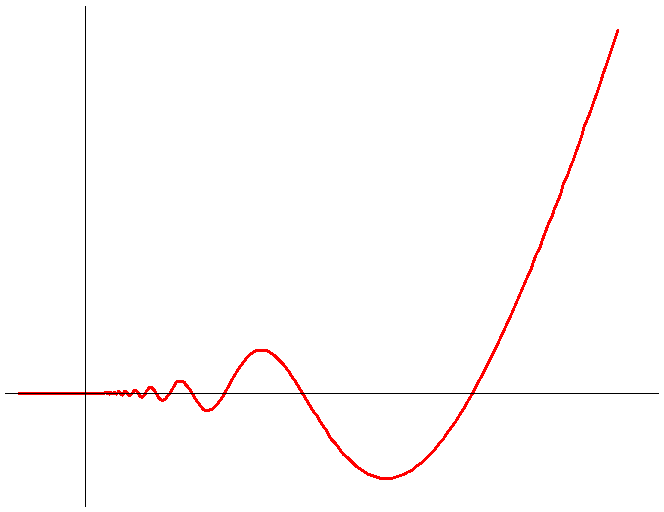}\quad\quad\quad
\includegraphics[width=0.38\textwidth]{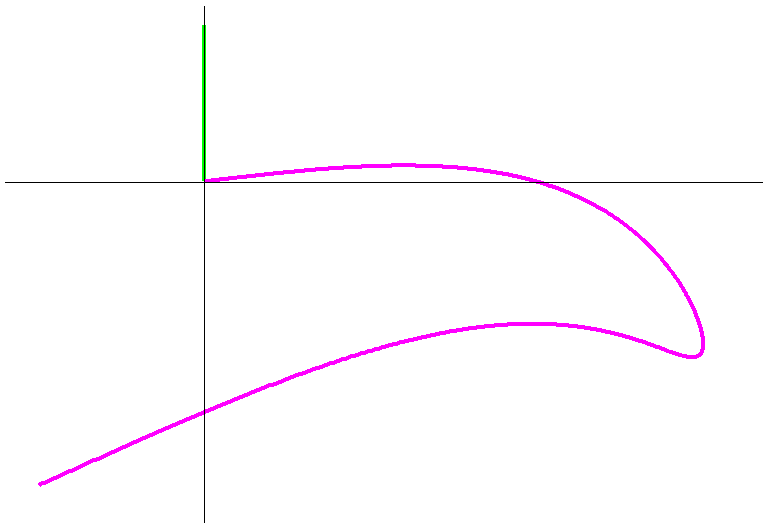}
\caption{\small{A numerical simulation obtained by setting
$I = \mathopen[0,1\mathclose]$, $a(x) = \sin(7\pi x)$
and $g(s) = s^{3}+s^{2}\sin(1/s)$. 
On the left we have shown the graph of $g(s)$. The nonlinearity $g(s)$ changes sign an infinite number of times in every neighborhood of zero.
On the right we have represented the image of the segment $\{0\}\times \mathopen[0,16\mathclose]$ through the Poincar\'{e} map in the phase-plane $(u,u')$.}}
\label{figure2}
\end{figure}

The plan of the paper is as follows. In Section~\ref{section2} we present some basic facts.
More in detail we list the hypotheses and we introduce an equivalent fixed point problem that permits to face the problem
with a topological approach. In fact, using the technical assumptions, we are able to compute the degree on suitable small and large balls,
in the same spirit of \cite{FeltrinZanolin2015}.

In Section~\ref{section3} we present our main result. The theorem we state is an immediate corollary 
of the results exhibited in the previous section. In particular, we prove that the topological degree is non-zero on an annular domain.
Therefore a nontrivial fixed point exists, this corresponds to a positive solution (using a standard maximum principle).
Straightforward corollaries are then obtained.

Section~\ref{section4} shows an important existence result of radially symmetric solutions on annular domains.

\section{Preliminaries}\label{section2}

In this section we state the hypotheses on $a(x)$ and on $g(s)$, we recall some classical results and we prove two preliminary lemmas
that are then employed in Section~\ref{section3} for the main result.

Consider the nontrivial closed interval $\mathopen{[}0,L\mathclose{]}$, pointing out that 
different choices of a nontrivial compact interval contained in $\mathbb{R}$ can be made. 
Let $a\colon\mathopen[0,L\mathclose]\to{\mathbb{R}}$ be a $L^{1}$-weight function. 
Clearly the case of a continuous function can be treated as well. We assume that
\begin{itemize}
\item[$(H1)$]\textit{there exist $m\geq 1$ intervals $I_{1},\ldots,I_{m}$, closed and pairwise disjoint, such that
\begin{equation*}
\begin{aligned}
   & a(x)\geq 0, \quad \text{for a.e. }  x\in \bigcup_{i=1}^{m} I_{i};
\\ & a(x)\leq 0, \quad \text{for a.e. }  x\in \mathopen{[}0,L\mathclose{]}\setminus \bigcup_{i=1}^{m} I_{i}.
\end{aligned}
\end{equation*}}
\end{itemize}
We underline that assumption $(H1)$ trivially includes the case where $a(x)\geq 0$ for a.e.~$x\in\mathopen{[}0,L\mathclose{]}$, 
taking $m=1$ and $I_{1}=\mathopen{[}0,L\mathclose{]}$.
As standard notation, we define
\begin{equation*}
a^{+}(x):=\max\{a(x),0\}, \qquad a^{-}(x):=\max\{-a(x),0\}.
\end{equation*}

Concerning the nonlinearity, we suppose that $g\colon{\mathbb{R}}^{+}\to{\mathbb{R}}$ is a continuous function such that
\begin{equation*}
g(0)=0 \quad \text{ and }\quad g\not\equiv0. \leqno{(H2)}
\end{equation*}
We set
\begin{equation*}
g_{0}^{inf}:=\liminf_{s\to 0^{+}} \dfrac{g(s)}{s} > -\infty,
\qquad
g_{0}^{sup}:=\limsup_{s\to 0^{+}} \dfrac{g(s)}{s} < +\infty
\end{equation*}
and
\begin{equation*}
g_{\infty}:=\liminf_{s\to +\infty} \dfrac{g(s)}{s} > 0.
\end{equation*}
We stress that  we do not suppose $g(s)\geq0$ on ${\mathbb{R}}^{+}$ and, in particular, it is not required that $g(s)>0$ for all $s>0$ 
(cf.~\cite{ErbeWang1994,FeltrinZanolin2015,GaudenziHabetsZanolin2003,LanWebb1998}).
Consequently, the nonlinearity $g(s)$ could be non-negative, non-positive or it could change sign, even an infinite number of times,
on a compact neighborhood of zero.

Now we show how the superlinearity of $g$ is expressed at zero and at infinity. 
As first step we impose a condition on the growth of $g(s)/s$ at $0$,  depending on the sign of $g(s)$.
Precisely we assume that
\begin{itemize}\item[$(H3)$] 
\begin{itemize}
\item[$\bullet$] \textit{if there exists $\delta>0$ such that $g(s)\geq0$, for all $s\in\mathopen{[}0,\delta\mathclose{]}$, it holds that 
\begin{equation*}
a^{+}(x)\not\equiv0 \text{ on } \mathopen{[}0,L\mathclose{]} \quad \text{ and } \quad g_{0}^{sup} < \lambda_{0}^{+},
\end{equation*}
where $\lambda_{0}^{+} > 0$ is the first eigenvalue of the eigenvalue problem
\begin{equation*}
\varphi'' + \lambda \, a^{+}(x) \, \varphi =0, \quad \varphi(0) = \varphi(L) = 0;
\end{equation*}
}
\item[$\bullet$] \textit{if there exists $\delta>0$ such that $g(s)\leq0$, for all $s\in\mathopen{[}0,\delta\mathclose{]}$, it holds that
\begin{equation*}
a^{-}(x)\not\equiv0 \text{ on } \mathopen{[}0,L\mathclose{]} \quad \text{ and } \quad g_{0}^{inf} > -\lambda_{0}^{-},
\end{equation*}
where $\lambda_{0}^{-} > 0$ is the first eigenvalue of the eigenvalue problem
\begin{equation*}
\varphi'' + \lambda \, a^{-}(x) \, \varphi =0, \quad \varphi(0) = \varphi(L) = 0;
\end{equation*}
}
\item[$\bullet$] \textit{if $g(s)$ changes sign an infinite number of times in every neighborhood of zero, it holds that
\begin{equation*}
a(x)\not\equiv0 \text{ on } \mathopen{[}0,L\mathclose{]} \quad \text{ and } \quad -\lambda_{0}<g_{0}^{inf}\leq g_{0}^{sup}<\lambda_{0},
\end{equation*}
where $\lambda_{0} > 0$ is the first eigenvalue of the eigenvalue problem
\begin{equation*}
\varphi'' + \lambda \, |a(x)| \, \varphi =0, \quad \varphi(0) = \varphi(L) = 0.
\end{equation*}
}
\end{itemize}
\end{itemize}
The functions $a(x)$ and $g(s)$ introduced in Figure~\ref{figure1} satisfy the first condition of hypothesis $(H3)$, 
while the example shown in Figure~\ref{figure2} corresponds to the third case.

As second step we define the superlinear behavior at infinity. We suppose that
\begin{itemize}
\item[$(H4)$] \textit{for all $i\in\{1,\ldots,m\}$
\begin{equation*}
a(x)\not\equiv 0 \text{ on } I_{i} \quad \text{ and } \quad g_{\infty} > \lambda_{1}^{i},
\end{equation*}
where $\lambda_{1}^{i} > 0$ is the first eigenvalue of the eigenvalue problem
\begin{equation*}
\varphi'' + \lambda \, a^{+}(x) \, \varphi =0, \quad \varphi|_{\partial I_{i}} = 0.
\end{equation*}
}
\end{itemize}

\medskip

Now we describe the topological approach we adopt to face problem \eqref{two-pointBVPag}. 
Our first goal is to introduce a completely continuous operator and to define an equivalent fixed point problem.

Let $\tilde{g}\colon{\mathbb{R}}\to{\mathbb{R}}$ be the standard extension of $g(s)$ defined as
\begin{equation*}
\tilde{g}(s)=
\begin{cases}
\, g(s), & \text{if } s\geq0; \\
\, 0,  & \text{if } s\leq0.
\end{cases}
\end{equation*}
We deal with the boundary value problem
\begin{equation}\label{BVP}
\begin{cases}
\, u''+a(x)\tilde{g}(u)=0 \\
\, u(0)=u(L)=0.
\end{cases}
\end{equation}
From conditions $(H2)$ and $(H3)$ and by a classical maximum principle (cf.~\cite{FeltrinZanolin2015,ManasevichNjokuZanolin1995}), 
it follows that all possible solutions of \eqref{BVP} are non-negative. 
Moreover, if these solutions are nontrivial, then they are strictly positive on 
$\mathopen{]}0,L\mathclose{[}$ and hence positive solutions of \eqref{two-pointBVPag}. 

The next step is to define the classical operator 
$\Phi\colon\mathcal{C}(\mathopen{[}0,L\mathclose{]})\to\mathcal{C}(\mathopen{[}0,L\mathclose{]})$ by
\begin{equation}\label{operator}
(\Phi u)(x):= \int_{0}^{L} G(x,\xi)a(\xi)\tilde{g}(u(\xi)) ~\!d\xi,
\end{equation}
where $G(x,s)$ is the Green function associated to the equation $u''+u=0$ with the two-point boundary condition.
The operator $\Phi$ is completely continuous in $\mathcal{C}(\mathopen{[}0,L\mathclose{]})$,
endowed with the $\sup$-norm $\|\cdot\|_{\infty}$, and such that $u$ is a fixed point of $\Phi$ if and only if $u$ is a solution of \eqref{BVP}.
Therefore we have transformed problem \eqref{two-pointBVPag} into an equivalent fixed point problem.

\medskip

We close this section by proving two technical lemmas that allow us to find a nontrivial fixed point of $\Phi$, 
hence a positive solution of \eqref{two-pointBVPag}.
The approach we use now is based on the Leray-Schauder topological degree and it is in the same spirit of \cite{FeltrinZanolin2015}.

Using this first lemma we are able to compute the degree of $Id-\Phi$ on small balls.

\begin{lemma}\label{lemmar0}
There exists $r_{0}>0$ such that
\begin{equation*}
deg(Id-\Phi,B(0,r),0)=1, \quad \forall \, 0<r\leq r_{0}.
\end{equation*}
\end{lemma}

\begin{proof}
We divide the proof in two steps.

\smallskip

\noindent
\textit{Step 1. We prove that there exists $r_{0}>0$ such that every solution $u(x)\geq 0$ of the two-point BVP
\begin{equation}\label{BVP1}
\begin{cases}
\, u''+ \vartheta a(x)g(u)=0, \quad 0\leq\vartheta\leq1, \\
\, u(0)=u(L)=0
\end{cases}
\end{equation}
satisfying $\max_{x\in \mathopen{[}0,L\mathclose{]}}u(x)\leq r_{0}$ is such that $u(x) = 0$, for all $x\in \mathopen{[}0,L\mathclose{]}$.
}

The proof of this first step is given only when there exists $\delta>0$ such that $g(s)\geq0$, for all $s\in\mathopen{[}0,\delta\mathclose{]}$.
The two remaining cases can be treated in an analogous way.

Using condition $(H3)$, we fix $0<r_{0}<\delta$ such that
\begin{equation*}
\dfrac{g(s)}{s}< \lambda_{0}^{+}, \quad \forall \, 0<s\leq r_{0}.
\end{equation*}
Now, suppose by contradiction that there exist $\vartheta\in\mathopen[0,1\mathclose]$
and a positive solution $u(x)\not\equiv 0$ of \eqref{BVP1} such that
$\max_{x\in \mathopen[0,L\mathclose]} u(x) = r$ for some $0 < r \leq r_{0}$.
The choice of $r_{0}$ and the maximum principle imply that
\begin{equation*}
0\leq\vartheta g(u(x)) < \lambda_{0}^{+}u(x), \quad \text{for all } x\in \mathopen]0,L\mathclose[.
\end{equation*}
Let $\varphi$ be a positive eigenfunction of
\begin{equation*}
\begin{cases}
\, \varphi''+ \lambda_{0}^{+}a^{+}(x)\varphi=0 \\
\, \varphi(0)= \varphi(L)=0.
\end{cases}
\end{equation*}
We stress that $\varphi(x)>0$, for all $x\in \mathopen]0,L\mathclose[$.
Using a Sturm comparison argument, we attain
\begin{equation*}
\begin{aligned}
0  & = \bigl{[}u'(x)\varphi(x)-u(x)\varphi'(x)\bigr{]}_{x=0}^{x=L}
\\ & = \int_{0}^{L}\dfrac{d}{dx}\Bigl{[}u'(x)\varphi(x)-u(x)\varphi'(x)\Bigr{]} ~\!dx
\\ & = \int_{0}^{L}\Bigl{[}u''(x)\varphi(x)-u(x)\varphi''(x)\Bigr{]} ~\!dx
\\ & = \int_{0}^{L}\Bigl{[}-\vartheta a(x)g(u(x))\varphi(x)+u(x)\lambda_{0}^{+}a^{+}(x)\varphi(x)\Bigr{]} ~\!dx
\\ & \geq \int_{0}^{L}\Bigl[\lambda_{0}^{+}u(x)-\vartheta g(u(x))\Bigr]a^{+}(x)\varphi(x) ~\!dx
\\ & > 0,
\end{aligned}
\end{equation*}
a contradiction.

\medskip

\noindent
\textit{Step 2. Computation of the degree.}
Let us fix $0\leq \vartheta \leq 1$. 
As remarked when we have introduced the operator $\Phi$, the maximum principle ensures that 
every fixed point in $\mathcal{C}(\mathopen[0,L\mathclose])$ of the operator $\vartheta\Phi$ is non-negative and, moreover,
$u\in\mathcal{C}(\mathopen[0,L\mathclose])$ satisfies $u=\vartheta\Phi(u)$ 
if and only if $u$ is a solution of the equation \eqref{BVP1}.
Therefore, setting $r\in\mathopen]0,r_{0}\mathclose]$, \textit{Step~1} implies that $\|u\|_{\infty} \neq r$ and hence
\begin{equation*}
u\neq \vartheta \Phi(u), \quad \forall \, \vartheta\in\mathopen[0,1\mathclose], \; \forall \, u\in \partial B(0,r).
\end{equation*}
By the homotopic invariance property of the topological degree, we obtain that
\begin{equation*}
deg(Id-\Phi,B(0,r),0)=deg(Id,B(0,r),0)=1.
\end{equation*}
\end{proof}

Now we compute the degree on large balls.

\begin{lemma}\label{lemmaR}
There exists $R^{*}>0$ such that
\begin{equation*}
deg(Id-\Phi,B(0,R),0)=0, \quad \forall \, R\geq R^{*}.
\end{equation*}
\end{lemma}

\begin{proof}
We divide the proof in two steps.

\smallskip

\noindent
\textit{Step 1. A priori bounds for $u$ on each $I_{i}$.}
For each $i\in\{1,\ldots,m\}$, we prove that there exists $R_{i}>0$ such that 
for each $L^{1}$-Carath\'{e}odory function $h\colon \mathopen[0,L\mathclose] \times{\mathbb{R}}^{+}\to {\mathbb{R}}$ with
\begin{equation*}
h(x,s)\geq a(x)g(s), \quad \text{a.e. } x\in I_{i}, \; \forall \, s\geq 0,
\end{equation*}
every solution $u(x)\geq0$ of the two-point BVP
\begin{equation}\label{BVP2}
\begin{cases}
\, u''+h(x,u)=0 \\
\, u(0)=u(L)=0
\end{cases}
\end{equation}
satisfies $\max_{x\in I_{i}}u(x)<R_{i}$.

We fix an index $i\in\{1,\ldots,m\}$ and set $I_{i}:=\mathopen[\sigma_{i},\tau_{i}\mathclose]$. 
Let $0<\varepsilon<(\tau_{i}-\sigma_{i})/2$ be fixed such that 
\begin{equation*}
a^{+}(x)\not\equiv0 \quad \text{on } I_{i}^{\varepsilon},
\end{equation*}
where $I_{i}^{\varepsilon}:=\mathopen[\sigma_{i}+\varepsilon,\tau_{i}-\varepsilon\mathclose]$,
and such that the first positive eigenvalue $\hat{\lambda}$ of the eigenvalue problem
\begin{equation}\label{eigen-pb}
\begin{cases}
\, \varphi''+ \lambda \, a^{+}(x) \, \varphi=0 \\
\, \varphi|_{\partial I_{i}^{\varepsilon}}=0
\end{cases}
\end{equation}
is such that
\begin{equation*}
0<\hat{\lambda}<g_{\infty}.
\end{equation*}
The existence of $\varepsilon$ is ensured by the continuity of the eigenvalue as function of the boundary condition 
(see~\cite{deFigueiredo1982,Zettl2005}) and by hypothesis $(H4)$. 
From the previous inequality it follows that there exists a constant $\tilde{R}>0$ such that
\begin{equation*}
g(s)>\hat{\lambda} s, \quad \forall \, s\geq \tilde{R}.
\end{equation*}

By contradiction, suppose there is not a constant $R_{i}>0$ with the properties listed above. So, for each integer $n>0$
there exists a solution $u_{n}\geq0$ of \eqref{BVP2} with $\max_{x\in I_{i}}u_{n}(x)=:\hat{R}_{n}>n$.

We claim that there exists an integer $N\geq\tilde{R}$ such that $u_{n}(x)>\tilde{R}$ for every $x\in I_{i}^{\varepsilon}$ and $n\geq N$.
If it is not true, for every integer $n\geq\tilde{R}$ there is an integer $\hat{n}\geq n$ and 
$x_{\hat{n}}\in I_{i}^{\varepsilon}$ such that $u_{\hat{n}}(x_{\hat{n}})=\tilde{R}$.
We note that the solution $u_{\hat{n}}(x)$ is concave on each subinterval of $I_{i}$ where $u_{\hat{n}}(x)\geq \tilde{R}$, 
since $a(x)g(s)\geq 0$ for a.e.~$x\in I_{i}$ and for all $s\geq\tilde{R}$.
Then, without loss of generality, we can assume that there exists a maximum point $\hat{x}_{\hat{n}}\in I_{i}$ of $u_{\hat{n}}$
such that $u_{\hat{n}}(x)>\tilde{R}$ for all $x$ between $x_{\hat{n}}$ and $\hat{x}_{\hat{n}}$
(if necessary, we change the choice of $x_{\hat{n}}$).
From the assumptions, it follows that
\begin{equation}\label{eq-u}
\hat{n}<\hat{R}_{\hat{n}}=u_{\hat{n}}(\hat{x}_{\hat{n}})
=u_{\hat{n}}(x_{\hat{n}})+\int_{x_{\hat{n}}}^{\hat{x}_{\hat{n}}}u_{\hat{n}}'(\xi)~\!d\xi
\leq \tilde{R}+(\tau_{i}-\sigma_{i}) |u_{\hat{n}}'(x_{\hat{n}})|.
\end{equation}
Since $h(x,s)$ is a $L^{1}$-Carath\'{e}odory function, there exists $\gamma_{\tilde{R}}\in L^{1}(\mathopen{[}0,L\mathclose{]},\mathbb{R}^{+})$
such that $|h(x,s)|\leq \gamma_{\tilde{R}}(x)$, for a.e.~$x\in\mathopen{[}0,L\mathclose{]}$ and for all $|s|\leq\tilde{R}$. Then, we fix a constant $C>0$ such that
\begin{equation*}
C>\dfrac{\tilde{R}}{\varepsilon} + \|\gamma_{\tilde{R}}\|_{L^{1}}.
\end{equation*}
Using \eqref{eq-u}, we have that for every $n \geq (\tau_{i}-\sigma_{i})C+\tilde{R}$ there exists $\hat{n}\geq n$ and 
$x_{\hat{n}}\in I_{i}^{\varepsilon}$ such that $u_{\hat{n}}(x_{\hat{n}})=\tilde{R}$ and $|u_{\hat{n}}'(x_{\hat{n}})| > C$.
Let us fix $n \geq (\tau_{i}-\sigma_{i})C+\tilde{R}$, $\hat{n}\geq n$ and $x_{\hat{n}}\in I_{i}^{\varepsilon}$ with the properties just listed.
Suppose that $u_{\hat{n}}'(x_{\hat{n}}) > C$ and consider the interval $\mathopen{[}\sigma_{i},x_{\hat{n}}\mathclose{]}$. 
If $u_{\hat{n}}'(x_{\hat{n}}) < -C$ we proceed similarly dealing with the interval $\mathopen{[}x_{\hat{n}},\tau_{i}\mathclose{]}$.
For every $x\in\mathopen{[}\sigma_{i},x_{\hat{n}}\mathclose{]}$
\begin{equation*}
u_{\hat{n}}'(x)=u_{\hat{n}}'(x_{\hat{n}})+\int_{x_{\hat{n}}}^{x}u_{\hat{n}}''(\xi)~\!d\xi,
\end{equation*}
then
\begin{equation*}
u_{\hat{n}}'(x)>C-\int_{x}^{x_{\hat{n}}} |h(\xi,u_{\hat{n}}(\xi))|~\!d\xi.
\end{equation*}
From this inequality we obtain that $u_{\hat{n}}(x)\leq\tilde{R}$, for all $x\in\mathopen{[}\sigma_{i},x_{\hat{n}}\mathclose{]}$,
and therefore
\begin{equation*}
u_{\hat{n}}'(x)>\dfrac{\tilde{R}}{\varepsilon}, \quad \text{ for all } x\in\mathopen{[}\sigma_{i},x_{\hat{n}}\mathclose{]}.
\end{equation*}
Then, we obtain
\begin{equation*}
\tilde{R} \leq \dfrac{\tilde{R}}{\varepsilon} (x_{\hat{n}}-\sigma_{i}) < \int_{\sigma_{i}}^{x_{\hat{n}}}u_{\hat{n}}'(\xi)~\!d\xi = 
u_{\hat{n}}(x_{\hat{n}})-u_{\hat{n}}(\sigma_{i}) \leq u_{\hat{n}}(x_{\hat{n}}) =\tilde{R},
\end{equation*}
a contradiction. 
Hence the claim is proved. So, we can fix an integer $N\geq\tilde{R}$ 
such that $u_{n}(x)>\tilde{R}$ for every $x\in I_{i}^{\varepsilon}$ and for $n\geq N$.

We denote by $\varphi$ the positive eigenfunction of the eigenvalue problem \eqref{eigen-pb} with $\|\varphi\|_{\infty}=1$.
Then $\varphi(x)>0$, for every $x\in\mathopen]\sigma_{i}+\varepsilon,\tau_{i}-\varepsilon\mathclose[$,
and $\varphi'(\sigma_{i}+\varepsilon)>0>\varphi'(\tau_{i}-\varepsilon)$.
We remark that $u_{n}(\sigma_{i}+\varepsilon)>0$ and $u_{n}(\tau_{i}-\varepsilon)>0$, for every integer $n$,
employing the maximum principle.

Using a Sturm comparison argument, for each $n\geq N$, we obtain
\begin{equation*}
\begin{aligned}
0  & > u_{n}(\tau_{i}-\varepsilon)\varphi'(\tau_{i}-\varepsilon)-u_{n}(\sigma_{i}+\varepsilon)\varphi'(\sigma_{i}+\varepsilon)
\\ & = \Bigl{[}u_{n}(x)\varphi'(x)-u'_{n}(x)\varphi(x)\Bigr{]}_{x=\sigma_{i}+\varepsilon}^{x=\tau_{i}-\varepsilon}
\\ & = \int_{\sigma_{i}+\varepsilon}^{\tau_{i}-\varepsilon}\dfrac{d}{dx}\Bigl{[}u_{n}(x)\varphi'(x)-u'_{n}(x)\varphi(x)\Bigr{]} ~\!dx
\\ & = \int_{I_{i}^{\varepsilon}}\Bigl{[}u_{n}(x)\varphi''(x)-u''_{n}(x)\varphi(x)\Bigr{]} ~\!dx
\\ & = \int_{I_{i}^{\varepsilon}}\Bigl{[}-u_{n}(x)\hat{\lambda}a^{+}(x)\varphi(x)+h(x,u_{n}(x))\varphi(x)\Bigr{]} ~\!dx
\\ & =  \int_{I_{i}^{\varepsilon}}\Bigl{[}h(x,u_{n}(x))-\hat{\lambda}a^{+}(x)u_{n}(x)\Bigr{]}\varphi(x) ~\!dx
\\ & \geq \int_{I_{i}^{\varepsilon}}\Bigl{[}a(x)g(u_{n}(x))-\hat{\lambda}a^{+}(x)u_{n}(x)\Bigr{]}\varphi(x) ~\!dx
\\ & = \int_{I_{i}^{\varepsilon}}\Bigl{[}g(u_{n}(x))-\hat{\lambda}u_{n}(x)\Bigr{]}a^{+}(x)\varphi(x) ~\!dx
\\ & \geq 0,
\end{aligned}
\end{equation*}
a contradiction.

\medskip

\noindent
\textit{Step 2. Computation of the degree.}
We stress that the constant $R_{i}$, $i\in\{1,\ldots,m\}$, does not depend on the function $h(x,s)$.
Define
\begin{equation*}
R^{*}:=\max_{i=1,\ldots,m}R_{i}+\tilde{R}>0
\end{equation*}
and fix a radius $R\geq R^{*}$.

We denote by $\mathbbm{1}_{A}$ the characteristic function of the set $A:= \bigcup_{i=1}^{m} I_{i}$.
Let us define $v(x):=\int_{I}G(x,s)\mathbbm{1}_{A}(s) ~\!ds$.
Using a classical result (see~\cite[Theorem~3.1]{deFigueiredo1982} or \cite[Lemma~1.1]{Nussbaum1973}), if we show that
\begin{equation}\label{eq-deF}
u \neq \Phi(u) + \alpha v,\quad \text{for all } u\in \partial B(0,R) \text{ and } \alpha \geq 0,
\end{equation}
the theorem is proved.

Let $\alpha \geq 0$. 
The maximum principle ensures that any nontrivial solution $u\in\mathcal{C}(\mathopen[0,L\mathclose])$ of 
$u =\Phi(u) + \alpha v$ is a non-negative solution of
$u'' + a(x)\tilde{g}(u) + \alpha \mathbbm{1}_{A}(x) = 0$ with $u(0) = u(L) = 0$. 
Hence, $u$ is a non-negative solution of \eqref{BVP2} with
\begin{equation*}
h(x,s) = a(x)g(s) + \alpha \mathbbm{1}_{A}(x).
\end{equation*}
By definition, we have that $h(x,s) \geq a(x)g(s)$, for a.e.~$x\in A$ and for all $s\geq 0$, and $h(x,s) = a(x)g(s)$,
for a.e.~$x\in \mathopen[0,L\mathclose]\setminus A$ and for all $s\geq 0$.
By the convexity of the solution $u$ on the intervals of $\mathopen[0,L\mathclose]\setminus A$ where $u(x)\geq\tilde{R}$,
we obtain that
\begin{equation*}
\|u\|_{\infty}=\max_{x\in \mathopen[0,L\mathclose]}u(x)\leq\max\Bigl{\{}\max_{x\in A}u(x), \tilde{R}\Bigr{\}}.
\end{equation*}
From \textit{Step 1} and the definition of $\tilde{R}$ we deduce that $\|u\|_{\infty} < R^{*} \leq R$. 
Then \eqref{eq-deF} is proved and the theorem follows.
\end{proof}

\section{The main result}\label{section3}

In this section we apply the two technical lemmas just proved to obtain the existence of a positive solution
to the two-point boundary value problem \eqref{two-pointBVPag}. 
More in detail, we use the additivity of the topological degree to provide the existence of a nontrivial fixed point 
of the operator $\Phi$ defined in \eqref{operator}.

A first immediate consequence of Lemma~\ref{lemmar0} and Lemma~\ref{lemmaR} is our main theorem. 

\begin{theorem}\label{MainTheorem}
Let $a\colon\mathopen[0,L\mathclose]\to{\mathbb{R}}$ be a $L^{1}$-function 
and $g\colon{\mathbb{R}}^{+}\to{\mathbb{R}}$ be a continuous function satisfying $(H1)$, $(H2)$, $(H3)$ and $(H4)$.
Then there exists at least a positive solution of the two-point boundary value problem \eqref{two-pointBVPag}.
\end{theorem}

\begin{proof}
Let $r_{0}$ be as in Lemma~\ref{lemmar0} and $R^{*}$ be as in Lemma~\ref{lemmaR}. We observe that $0<r_{0}<R^{*}<+\infty$. 
From the additivity property and the two preliminary lemmas it follows that
\begin{equation*}
\begin{aligned}
& deg(Id-\Phi,B(0,R^{*})\setminus B[0,r_{0}],0)=
\\ & = deg(Id-\Phi,B(0,R^{*}),0)-deg(Id-\Phi,B(0,r_{0}),0)=
\\ & =0-1=-1\neq0.
\end{aligned}
\end{equation*}
Then there exists a nontrivial fixed point of $\Phi$ and hence a corresponding positive solution of \eqref{two-pointBVPag}, as already remarked.
\end{proof}

{}From Theorem~\ref{MainTheorem} we easily achieve the following two results.

\begin{corollary}\label{cor1}
Let $a\colon\mathopen[0,L\mathclose]\to{\mathbb{R}}$ be a $L^{1}$-function and $g\colon{\mathbb{R}}^{+}\to{\mathbb{R}}$ be a continuous function satisfying $(H1)$ and $(H2)$. Assume that
\begin{equation*}
g'(0)=\lim_{s\to0^{+}}\dfrac{g(s)}{s}=0,
\end{equation*}
and, for each $i\in\{1,\ldots,m\}$,  suppose that $a(x)\not\equiv 0$ on $I_{i}$ and 
\begin{equation*}
g'(\infty):=\lim_{s\to+\infty}\dfrac{g(s)}{s}=+\infty.
\end{equation*}
Then there exists at least a positive solution of the two-point BVP \eqref{two-pointBVPag}.
\end{corollary}

\begin{corollary}\label{cor2}
Let $a\colon\mathopen[0,L\mathclose]\to{\mathbb{R}}$ be a $L^{1}$-function satisfying $(H1)$ and such that 
$a(x)\not\equiv 0$ on $I_{i}$, for each $i\in\{1,\ldots,m\}$.
Let $g\colon{\mathbb{R}}^{+}\to{\mathbb{R}}$ be a continuous function satisfying $(H2)$ and such that $g'(0)=0$ and $g'(\infty)=\Lambda>0$.
Then there exists $\lambda^{*}>0$ such that, for each $\lambda>\lambda^{*}$, the two-point BVP
\begin{equation*}
\begin{cases}
\, u''+\lambda a(x)g(u)=0 \\
\, u(0)=u(L)=0
\end{cases}
\end{equation*}
has at least a positive solution.
\end{corollary}

Although hypothesis $(H1)$ is more interesting when the set $\mathopen[0,L\mathclose]\setminus \bigcup_{i=1}^{m} I_{i}$ is not negligible,
we can consider a weight $a(x)\geq0$ for a.e.~$x\in\mathopen[0,L\mathclose]$, as previously observed. In that situation Corollary \ref{cor1}
ensures the existence of a positive solution in the superlinear case (i.e.~$g'(0)=0$ and $g'(\infty)=+\infty$), provided that $a\not\equiv 0$. 
No sign condition on the function $g(s)$ is required. Thus we have extended \cite[Theorem~1]{ErbeWang1994}, 
attained as an application of Krasnosel'ski\u{\i} fixed point Theorem.

\begin{remark}
Our approach is based on the definition of a fixed point problem which is equivalent to the boundary value problem considered. 
It is clear that we could deal with different conditions at the boundary of $\mathopen{[}0,L\mathclose{]}$ like 
$u'(0)=u(L)=0$ or $u(0)=u'(L)=0$, since a suitable maximum principle and a Green function (cf.~\cite{ErbeWang1994})
are available to define an equivalent fixed point problem and to adapt the scheme shown in this paper.
\end{remark}

\section{Radially symmetric solutions}\label{section4}

We denote by $\|\cdot\|$ the Euclidean norm in ${\mathbb{R}}^{N}$ (for $N \geq 2$). Let
\begin{equation*}
\Omega:= B(0,R_{2})\setminus B[0,R_{1}] = \{x\in {\mathbb{R}}^{N} \colon R_{1} < \|x\| < R_{2}\}
\end{equation*}
be an open annular domain, with $0 < R_{1} < R_{2}$.
Let $a\colon \mathopen{[}R_{1},R_{2}\mathclose{]}\to {\mathbb{R}}$ be a continuous function.
In this section we consider the Dirichlet boundary value problem
\begin{equation}\label{eq-dir}
\begin{cases}
\, -\Delta \,u = a(\|x\|)\,g(u) & \text{ in } \Omega \\
\, u = 0 & \text{ on } \partial\Omega
\end{cases}
\end{equation}
and we are interested in the existence of positive solutions of \eqref{eq-dir},
namely classical solutions such that $u(x) > 0$ for all $x\in \Omega$.

Since we look for radially symmetric solutions of \eqref{eq-dir}, our study can be reduced to the search of positive solutions
of the two-point boundary value problem
\begin{equation}\label{eq-rad}
w''(r) + \dfrac{N-1}{r}w'(r) + a(r) g(w(r)) = 0, \quad w(R_{1}) = w(R_{2}) = 0.
\end{equation}
Indeed, if $w(r)$ is a solution of \eqref{eq-rad}, then $u(x):= w(\|x\|)$ is a solution of \eqref{eq-dir}.
Using the standard change of variable
\begin{equation*}
t = h(r):= \int_{R_{1}}^r \xi^{1-N} ~\!d\xi
\end{equation*}
and defining
\begin{equation*}
L:= \int_{R_{1}}^{R_{2}} \xi^{1-N} ~\!d\xi, \quad r(t):= h^{-1}(t) \quad \text{and} \quad v(t)=w(r(t)),
\end{equation*}
we transform \eqref{eq-rad} into the equivalent problem
\begin{equation}\label{eq-rad1}
v''(t) +  r(t)^{2(N-1)}a(r(t)) g(v(t)) = 0, \quad v(0) = v(L) = 0.
\end{equation}
Consequently, the two-point boundary value problem \eqref{eq-rad1} is of the same form of \eqref{two-pointBVPag} 
considering $r(t)^{2(N-1)} a(r(t))$ as weight function.

Clearly the following result holds.

\begin{theorem}\label{th-radial}
Let $a\colon\mathopen{[}R_{1},R_{2}\mathclose{]}\to{\mathbb{R}}$ and $g\colon{\mathbb{R}}^{+}\to{\mathbb{R}}$ 
be continuous functions satisfying $(H1)$, $(H2)$, $(H3)$ and $(H4)$.
Then problem \eqref{eq-dir} has at least a positive solution.
\end{theorem}

\section*{Acknowledgments} 

This work benefited from long enlightening discussions, helpful suggestions and encouragement of Fabio Zanolin. 
This research was supported by \textit{SISSA - International School for Advanced Studies} and \textit{Universit\`{a} degli Studi di Udine}.

\bibliographystyle{elsart-num-sort}
\bibliography{Feltrin_biblio}

\bigskip
\begin{flushleft}

{\small{\it Preprint}}

{\small{\it September 2014}}

\end{flushleft}

\end{document}